\documentclass[12pt,reqno]{article}

\usepackage[usenames]{xcolor}
\usepackage{amssymb}
\usepackage{amsmath}
\usepackage{amsthm}
\usepackage{amsfonts}

\definecolor{webgreen}{rgb}{0,.5,0}
\definecolor{webbrown}{rgb}{.6,0,0}

\usepackage[colorlinks=true,
linkcolor=webgreen,
filecolor=webbrown,
citecolor=webgreen]{hyperref}

\theoremstyle{plain}
\newtheorem{theorem}{Theorem}
\newtheorem{corollary}[theorem]{Corollary}
\newtheorem{lemma}[theorem]{Lemma}
\newtheorem{proposition}[theorem]{Proposition}

\newcounter{remarkno}
\theoremstyle{remark}
\newtheorem{remark}[remarkno]{Remark}

\def\C{\ensuremath{\mathbb C}}
\def\bt{\ensuremath{\mathbf t}}
\def\bs{\ensuremath{\mathbf s}}
\def\bx{\ensuremath{\mathbf x}}
\newcommand{\be}{\begin{equation}}
\newcommand{\ee}{\end{equation}}
\DeclareMathOperator{\Cyc}{Cyc}
\newcommand{\order}[1]{\ensuremath{{\mathcal O}\left(#1\right)}}
\renewcommand{\c}[1]{\ensuremath{\overline{#1}}}
\newcommand{ \highl}[1]{{\colorbox{gray}{\hskip -3pt #1 \hskip -5pt}}}
\newcommand{\seqnum}[1]{\href{https://oeis.org/#1}{\rm \underline{#1}}}

\begin{document}

\begin{center}
\vskip 1cm{\LARGE\bf 
On Longest Increasing Subsequences in Words in which All Multiplicities are Equal
\vskip .1in
}
\vskip 1cm
\large
Ferenc Balogh\\ 
Department of Mathematics\\
John Abbott College\\
21275 Lakeshore Road\\
Sainte-Anne-de-Bellevue, PQ H9X 3L9\\
Canada\\
\href{mailto:ferenc.balogh@johnabbott.qc.ca}{\tt ferenc.balogh@johnabbott.qc.ca}  \\
\end{center}

\begin{abstract} Gessel's famous Bessel determinant formula gives the generating function of the number of permutations without increasing subsequences of a given length. Ekhad and Zeilberger proposed the challenge of finding a suitable generalization for permutations of multisets in which all multiplicities are equal, that is, to count words of length $rn$ from an alphabet consisting of $n$ letters in which each letter appears exactly $r$ times and which have no increasing subsequences of length $d$.

In this paper we present such a generating function expressible as a multiple integral of the product of a Gessel-type Toeplitz determinant with the exponentiated cycle index polynomial of the symmetric group on $r$ elements.
\end{abstract}

\section{Introduction and statement of results}
For positive integers $n$, $d$, and $r$, let $A_{d,r}(n)$ denote the number of words from the alphabet $\{1,2,\ldots, n\}$ of length $rn$ such that
\begin{itemize}
\item each letter appears exactly $r$ times,
\item  there are no (strictly) increasing subsequences of length $d$.
\end{itemize}
For example, out of the $\frac{6!}{2!2!2!}=90$ words from the alphabet $\{1,2,3\}$ in which each letter has multiplicity $2$, there are $43$ words with no increasing subsequences of length $3$ (shown highlighted in Table \ref{table:illustration}).

The goal of this paper is to find generating functions for the numbers $(A_{d,r}(n))_{n \geq 0}$, with $d$ and $r$ fixed.

Note that the special case $r=1$ (each letter appears only once) corresponds to the problem of determining the length of the longest increasing subsequence of a random permutation. For a detailed history of this problem and its many beautiful connections with representation theory, analysis, and probability theory, see Romik's book \cite{Ro}.

Our starting point is the following observation by Ekhad and Zeilberger \cite{EZ}: the Robinson-Schensted-Knuth (RSK) correspondence \cite{St} implies that $A_{d+1,r}(n)$ can be written as a sum over partitions $\lambda$ of weight $|\lambda|$ equal to $rn$ and of length $\ell(\lambda)$ at most $d$ in the following form:
\be
A_{d+1,r}(n) = \sum_{\substack{|\lambda|=rn\\ \ell(\lambda)\leq d}}f_{\lambda}g_{\lambda}^{(r)},
\label{RSK}
\ee
where $f_{\lambda}$ is the number of standard Young tableaux of shape $\lambda$, and $g_{\lambda}^{(r)}$ is the number of semi-standard Young tableaux of shape $\lambda$ whose content is $1^{r}2^{r}\cdots n^{r}$. Note that the original formula of Ekhad and Zeilberger \cite{EZ} had a misprint that was found and corrected by Chapuy \cite{Ch}.

For $r=1$ the quantity $g_{\lambda}^{(1)}$ reduces to $f_{\lambda}$ and hence
\be
A_{d+1,1}(n) = \sum_{\substack{|\lambda|=n\\ \ell(\lambda)\leq d}}f_{\lambda}^2.
\ee
Gessel \cite{Ge} found the following remarkable generating function for the sequence $(A_{d+1,1}(n))_{n\geq 0}$:
\be
\label{Gessel_original}
\sum_{n=0}^{\infty}\frac{A_{d+1,1}(n)}{(n!)^2}x^{2n} = \det(G_d(x)),
\ee
where $G_d(x)$ is the $d\times d$ Toeplitz matrix
\be
(G_d(x))_{ij}=I_{|i-j|}(2x), \quad 1\leq i,j \leq d,
\label{Gessel_Toeplitz}
\ee
and the entries $I_m(u)$ are \emph{modified Bessel functions}  with integer parameters. The functions $I_m(u)$ are given by the generating function \cite[Eq.\ 10.35.1]{DLMF}.
\be
\sum_{m=-\infty}^{\infty}I_m(u)z^{m} = \exp\left(\frac{u}{2}\left(\frac{1}{z}+z\right)\right).
\ee
\begin{center}
\begin{table}[htb!]
\renewcommand{\arraystretch}{1.2}
{\tiny
\begin{tabular}{cccccccccc}
1 1 2 2 3 3 & 1 1 2 3 2 3 & 1 1 2 3 3 2 & 1 1 3 2 2 3 & 1 1 3 2 3 2 &  \highl{1 1 3 3 2 2} & 1 2 1 2 3 3 & 1 2 1 3 2 3 & 1 2 1 3 3 2 & 1 2 2 1 3 3 \cr
{1 2 2 3 1 3} & 1 2 2 3 3 1 & 1 2 3 1 2 3 & 1 2 3 1 3 2 & 1 2 3 2 1 3 & 1 2 3 2 3 1 & 1 2 3 3 1 2 & 1 2 3 3 2 1 & 1 3 1 2 2 3 & 1 3 1 2 3 2 \cr
  \highl{1 3 1 3 2 2} & 1 3 2 1 2 3 & 1 3 2 1 3 2 & 1 3 2 2 1 3 & 1 3 2 2 3 1 & 1 3 2 3 1 2 & 1 3 2 3 2 1 &  \highl{1 3 3 1 2 2} &  \highl{1 3 3 2 1 2} &  \highl{1 3 3 2 2 1} \cr
  2 1 1 2 3 3 & 2 1 1 3 2 3 &  \highl{2 1 1 3 3 2} & 2 1 2 1 3 3 & 2 1 2 3 1 3 & 2 1 2 3 3 1 & 2 1 3 1 2 3 &  \highl{2 1 3 1 3 2} & 2 1 3 2 1 3 & 2 1 3 2 3 1 \cr
    \highl{2 1 3 3 1 2} &  \highl{2 1 3 3 2 1} &  \highl{2 2 1 1 3 3} &  \highl{2 2 1 3 1 3} &  \highl{2 2 1 3 3 1} &  \highl{2 2 3 1 1 3} &  \highl{2 2 3 1 3 1} &  \highl{2 2 3 3 1 1} & 2 3 1 1 2 3 &  \highl{2 3 1 1 3 2} \cr
    2 3 1 2 1 3 & 2 3 1 2 3 1 &  \highl{2 3 1 3 1 2} &  \highl{2 3 1 3 2 1} &  \highl{2 3 2 1 1 3} &  \highl{2 3 2 1 3 1} &  \highl{2 3 2 3 1 1} &  \highl{2 3 3 1 1 2} &  \highl{2 3 3 1 2 1} &  \highl{2 3 3 2 1 1} \cr
     3 1 1 2 2 3 & 3 1 1 2 3 2 &  \highl{3 1 1 3 2 2} & 3 1 2 1 2 3 & 3 1 2 1 3 2 & 3 1 2 2 1 3 & 3 1 2 2 3 1 & 3 1 2 3 1 2 & 3 1 2 3 2 1 &  \highl{3 1 3 1 2 2} \cr
       \highl{3 1 3 2 1 2} &  \highl{3 1 3 2 2 1} & 3 2 1 1 2 3 &  \highl{3 2 1 1 3 2} & 3 2 1 2 1 3 & 3 2 1 2 3 1 &  \highl{3 2 1 3 1 2} &  \highl{3 2 1 3 2 1} &  \highl{3 2 2 1 1 3} &  \highl{3 2 2 1 3 1} \cr
        \highl{3 2 2 3 1 1} &  \highl{3 2 3 1 1 2} &  \highl{3 2 3 1 2 1} &  \highl{3 2 3 2 1 1} &  \highl{3 3 1 1 2 2} &  \highl{3 3 1 2 1 2} &  \highl{3 3 1 2 2 1} &  \highl{3 3 2 1 1 2} &  \highl{3 3 2 1 2 1} &  \highl{3 3 2 2 1 1}
\end{tabular}
}
\caption{This $10\times 9$ lexicographic array lists the words from the alphabet $\{1,2,3\}$ with all multiplicities equal to $2$. The highlighted words have no increasing subsequences of length $3$, there are $A_{3,2}(3)=43$ of them.}
\label{table:illustration}
\end{table}
\end{center}
\renewcommand{\arraystretch}{1}

The determinant formula \eqref{Gessel_original} played a crucial role in the subsequent developments of the asymptotic analysis of large permutations, most notably it lead to the famous Baik--Deift--Johansson theorem \cite{BDJ}. There are many related results for increasing subsequences of words (also referred to as permutations of multisets) \cite{AD,BR, BOO,ITW1,ITW2,vM,TW}.

Ekhad and Zeilberger \cite{EZ}  proposed the challenge to generalize \eqref{Gessel_original} for multiple occurrences $r>1$. This goal is fully accomplished in the main result of this paper.

For the sake of comparison with the general formulae to be presented, one can write Gessel's generating function \eqref{Gessel_original} in a slightly different Toeplitz determinant form:
\be
\label{Gessel_modified}
\sum_{n=0}^{\infty}\frac{A_{d+1,1}(n)}{(n!)^2}(v \c{v})^{n} = \det (\tilde G_d(x)),
\ee
where
\be
(\tilde G_d(x))_{ij}=\varphi_{i-j}(v,\bar{v}), \quad 1\leq i,j \leq d,
\label{Gessel_Toeplitz_2}
\ee
and the entries $\varphi_m(v,\bar{v})$ are the Laurent series coefficients at $z=0$ of the formal Toeplitz symbol
\be
\varphi(v,\bar{v};z)=\sum_{m=-\infty}^{\infty}\varphi_m(v,\bar{v}) z^m = \exp\left(\frac{v}{z}+\c{v}z\right).
\label{symbol:Gessel}
\ee
The variables $v$ and $\c{v}$ above are complex-valued bookkeeping indeterminates. Following common practice in complex variables, $v$ and $\c{v}$ are treated as independent variables, and $\c{v}$ is not identified by the complex conjugate of $v$ by default. Whenever the context dictates that we make $\c{v}$ equal to the complex conjugate of $v$, this will be stated explicitly. 

The Laurent series coefficients $\varphi_m(v,\c{v})$ in \eqref{symbol:Gessel} are well-defined: by splitting the exponent of the right hand side of \eqref{symbol:Gessel} into positive and negative parts, taking the exponentials separately, and collecting the $z$ terms results in a meaningful formal power series in $v$ and $\c{v}$:
\be
\varphi_m(v,\c{v}) = [z^m]\exp\left(\frac{v}{z}\right)\exp\left(\c{v}z\right)=[z^m]\sum_{k,l=0}^{\infty}\frac{v^k\c{v}^l z^{l-k} }{k!l !}
=\begin{cases}
\displaystyle \sum_{j=0}^{\infty}\frac{v^j\c{v}^{j+m}}{j!(j+m)!}, & \text{if $m>0$;}\\
\displaystyle \sum_{j=0}^{\infty}\frac{v^j\c{v}^j}{(j!)^2}, & \text{if $m=0$;}\\
\displaystyle \sum_{j=0}^{\infty}\frac{v^{j-m}\c{v}^j}{(j-m)!j!}, & \text{if $m<0$.}
\end{cases}
\ee
It is easy to see that $\varphi_m(x,x) = I_{|m|}(2x)$, and therefore \eqref{Gessel_modified} reduces to \eqref{Gessel_original} with $v=\c{v}=x$.

Before stating the main result, recall that the \emph{cycle index polynomial} of the symmetric group $S_r$ is the multivariate polynomial
\be
\label{cyc}
\Cyc_r(t_1,\ldots, t_r)=\frac{1}{r!}\sum_{|\mu|=r}\left|C_{\mu}\right|\prod_{i=1}^{\ell(\mu)}t_{\mu_i}
\ee
in the indeterminates $t_1,\ldots, t_r$, where the sum is over partitions $\mu$ of weight $|\mu|=r$, $C_{\mu}\subset S_{r}$ is the conjugacy class consisting of permutations of cycle type $\mu$ and $\ell(\mu)$ stands for the length (number of parts) of the partition $\mu$. For example, for $r=1,2,3$ the cycle index polynomials are
\begin{align}
\Cyc_1(t_1) &= t_1,\cr
\Cyc_2(t_1,t_2) &= \frac{1}{2}\left(t_1^2+t_2\right),\cr
\Cyc_3(t_1,t_2,t_3) &= \frac{1}{6}\left(t_1^3+3 t_1^2t_2+2 t_3\right).
\end{align}
Note that the cycle index polynomial $\Cyc_r$ is a weighted homogeneous polynomial of degree $r$ in the natural grading
$\deg(t_j) = j$. In what follows, we use the abbreviated notation $\bt = (t_1,t_2,\ldots)$.

\begin{theorem}
\label{thethm}
The sequence $(A_{d+1,r}(n))_{n \geq 0}$ has the formal series generating function
\be
\label{main_formula}
\sum_{n=0}^{\infty}\frac{A_{d+1,r}(n)}{(rn)!n!}(v\c{v})^{rn} = \frac{1}{\pi^{r}}\underbrace{\int_{\C}\cdots\int_{\C}}_{r}\det(T_d(v,\bt))\cdot \exp\left(\bar{v}^r\overline{\Cyc_{r}(\bt)}\right)\prod_{j=1}^{r}e^{-|t_j|^2}dA(t_j),
\ee
where $T_d(v,\bt)$ is the Toeplitz matrix
\be
(T_d(v,\bt))_{ij} =\varphi_{i-j}(v,\bt), \quad 1 \leq i, j \leq d,
\ee
associated with the formal Toeplitz symbol
\be
\varphi(v,\bt;z)=\sum_{m=-\infty}^{\infty}\varphi_{m}(v,\bt)z^m=\exp\left(\frac{v}{z}+ \sum_{j=1}^{r}t_j z^j\right),
\ee
and $dA$ stands for the area measure in the complex plane.
\end{theorem}
\begin{remark} In terms of the independent bookkeeping variables $v$ and $\bar{v}$,  the Toeplitz determinant $\det(T_d(v,\bt))$ and the exponentiated cycle index polynomial $\exp\left(\c{v}^r\overline{\Cyc_{r}(\bt)}\right)$ can both be expanded into formal power series, and the coefficients of the resulting product series
\[
\det(T_d(v,\bt))\exp\left(\c{v}^r\overline{\Cyc_{r}(\bt)}\right)= \sum_{k,l=0}^{\infty}F_k(\bt)\overline{G_l(\bt)}v^{k}\bar{v}^{l}
\]
are products of weighted homogeneous polynomials in $\bt$ and $\overline{\bt}$ and therefore their integrals with respect to multiple planar Gaussian measures in \eqref{main_formula} are well-defined.
\end{remark}
\begin{remark} I am grateful to M.\ Bertola for the following observation:
\begin{align}
\nonumber
\frac{1}{\pi^{r}}\underbrace{\int_{\C}\cdots\int_{\C}}_{r}t_1^{k_1}\cdots t_r^{k_r}\cdot \overline{t_1^{l_1}\cdots t_r^{l_r}}\prod_{j=1}^{r}e^{-|t_j|^2}dA(t_j) &=\prod_{j=1}^{r}\delta_{k_j l_j} k_j !\\
&= \langle t_1^{k_1}\cdots t_r^{k_r} , t_1^{l_1}\cdots t_r^{l_r}\rangle,
\end{align}
where $\langle\cdot,\cdot\rangle$ is Macdonald's scalar product on symmetric functions \cite[Eq.\ (4.7)]{Mac}, written in terms of the normalized power sum variables $t_i=p_i/i$. Therefore we can rephrase \eqref{main_formula} as
\be
\label{main_formula_scalar}
\sum_{n=0}^{\infty}\frac{A_{d+1,r}(n)}{(rn)!n!}(v\c{v})^{rn} = \left\langle \det(T_{d}(v,\bt)), \exp\left(\bar{v}^r\Cyc_{r}(\bt)\right)\right\rangle.
\ee
\end{remark}
\begin{remark} The last integral with respect to $t_r$ in \eqref{main_formula} can be replaced by the direct evaluation
\be
\label{eval_t_r}
\bar{t}_r=\frac{\bar{v}^r}{r}.
\ee
This follows from the integral identity
\be
\frac{1}{\pi}\int_{\C} P(t_r) \exp\left(\frac{\bar{v}^r}{r}\bar{t}_r - |t_r|^2\right)dA(t_r) = P\left(\frac{\bar{v}^{r}}{r}\right),
\ee
valid for all polynomials $P(t)$.
The $1/r$ factor in \eqref{eval_t_r} comes from
\be
[t_r]\Cyc_r(\bt)=\frac{1}{r!}\left|C_{(r)}\right| = \frac{(r-1)!}{r!}=\frac{1}{r}.
\ee
In particular, the generating function \eqref{main_formula} for $r=1$ reduces to Gessel's Toeplitz determinant formula \eqref{Gessel_modified} in terms of Bessel functions.
\end{remark}
\begin{remark}
For the case of double occurrences $r=2$, there is a \emph{quadrature identity}
\be
\frac{1}{\pi^2}\int_{\C}\int_{\C} Q(t_1,t_2) e^{\frac{\bar{v}^2}{2}\left(\bar{t}_1^2+\bar{t}_2\right) - |t_1|^2-|t_2|^2}dA(t_1)dA(t_2)=
\frac{1}{\sqrt{2\pi}}\int_{-\infty}^{\infty}Q\left(\bar{v}x,\frac{\bar{v}^2}{2}\right)e^{-\frac{x^2}{2}}dx
\ee
for polynomials $Q(t_1,t_2)$, assuming that $\left|\bar{v}^2/2\right|< 1$. Therefore, for $r=2$ our general formula \eqref{main_formula} reduces to
\be
\sum_{n=0}^{\infty}\frac{A_{d+1,2}(n)}{(2n)!n!}(v\c{v})^{2n} = \frac{1}{\sqrt{2\pi}}\int_{-\infty}^{\infty}\det(\hat T_d(v,\c{v},x))e^{-\frac{x^2}{2}}dx,
\label{r_2}
\ee
where $\hat T_d(v,\c{v},\bt)$ is the $d\times d$ Toeplitz determinant associated with the symbol
\be
\varphi(v,\bar{v},x;z)=\exp\left(\frac{v}{z}+ \bar{v}xz+\frac{\bar{v}^2}{2}z^2\right).
\ee
By setting $\bar{v}=1$ one can further simplify \eqref{r_2} to the integral involving the Toeplitz determinant of the simplified symbol
\be
\varphi(v,x;z)=\exp\left(\frac{v}{z}+ xz+\frac{1}{2}z^2\right).
\ee
\end{remark}
\begin{remark} The numbers $(A_{d+1,r}(n))_{n \geq 0}$ can be given a natural probabilistic interpretation described as follows. Consider the uniform measure on the set of words of length $rn$ in $n$ letters with each letter appearing exactly $r$ times, and let $L_{n,r}(w)$ denote the length of the longest increasing subsequence in a random word $w$, interpreted itself as a random variable. By the definition of $A_{d+1,r}(n)$ we have
\be
\operatorname{Prob}\left(L_{n,r}(w)\leq d\right) = \frac{A_{d+1,r}(n)}{\frac{(rn)!}{(r!)^n}},
\ee
and hence a slightly adapted version of \eqref{main_formula} can be interpreted as the ``Poissonization'' of the distributions of $L_{n,r}$:
\begin{multline}
\label{Poissonized}
e^{-\frac{|v|^{2r}}{r!}}\sum_{n=0}^{\infty}\operatorname{Prob}\left(L_{n,r}(w)\leq d\right)\frac{1}{n!}\left(\frac{|v|^{2r}}{r!}\right)^{n}\\
=\frac{e^{-\frac{|v|^{2r}}{r!}}}{\pi^{r}}\underbrace{\int_{\C}\cdots\int_{\C}}_{r}\det(T_d(v,\bt))\cdot \exp\left(\bar{v}^r\overline{\Cyc_{r}(\bt)}\right)\prod_{j=1}^{r}e^{-|t_j|^2}dA(t_j).
\end{multline}
Note that it is natural here to identify $\c{v}$ with the complex conjugate of $v$ in \eqref{Poissonized}.

In this probabilistic context, the formula \eqref{Poissonized} may turn out to be useful in a subsequent asymptotic analysis of  the probabilities $\operatorname{Prob}\left(L_{n,r}(w)\leq d\right)$ as $n \to \infty$ and $d \to \infty$ simultaneously, possibly via the Toeplitz--Fredholm techniques of Borodin, Okounkov, and Olshanski \cite{BOO}, combined with Johansson's de-Poissonization lemma \cite{J}.
\end{remark}

{\bf Outline.} In Sec.\ \ref{sec:schur} an integral representation of $g^{(r)}_{\lambda}$ is given in terms of the Schur polynomial $s_{\lambda}$ associated with the partition $\lambda$, which in turn leads to an integral formula for $A_{d+1,r}(n)$.
 Gessel's general Toeplitz determinant for length-restricted Cauchy--Littlewood-type sums over partitions is recalled in Sec.\ \ref{sec:gessel} and it is employed to complete the proof of Theorem \ref{thethm}. In Sec.\ \ref{sec:calculations} we illustrate how the abstract generating function can be used in practice to calculate $A_{d+1,r}(n)$ for small values of $d$, $r$, and $n$ by using standard symbolic computer algebra software.

\section{The Kostka coefficients $g_{\lambda}^{(r)}$ as integrals of Schur polynomials}
\label{sec:schur}
Following the standard approach in the theory of symmetric functions \cite{Mac}, one defines the Schur polynomials $s_{\lambda} \in \C[t_1,t_2,t_3,\ldots]$ 
in the normalized power sum variables $\bt = (t_1,t_2,\ldots)$ via the Jacobi-Trudi formula
\be
s_{\lambda}(\bt) =\det \left(h_{\lambda_i-i+j}(\bt)\right)_{1\leq i,j \leq \ell(\lambda)},
\ee
where the polynomials $h_k(\bt)$ (the complete symmetric polynomials) are obtained from the generating function
\be
\sum_{k=0}^{\infty}h_k(\bt)z^k = \exp\left(\sum_{j=1}^{\infty}t_j z^j\right),
\ee
with $h_k(\bt)=0$ for negative values of $k$. The polynomial $h_k(\bt)$ is a weighted homogeneous polynomial of degree $k$ in the natural grading $\deg(t_j) = j$.
Therefore $s_{\lambda}(\bt)$ is also a weighted homogeneous polynomial of degree $|\lambda|$ and hence it has the scaling property
\be
s_{\lambda}\left(a t_1,a^2t_2,a^3t_3,\ldots\right) = a^{|\lambda|}s_{\lambda}(t_1,t_2,t_3,\ldots).
\ee
It is well-known \cite{Mac} that the number of standard Young tableaux of shape $\lambda$ is given by the special evaluation
\be
f_{\lambda} = |\lambda|! s_{\lambda}\left(1,0,0,\ldots\right).
\label{f_lambda_schur}
\ee
As pointed out by Chapuy \cite{Ch}, the coefficient $g_{\lambda}^{(r)}$ is the Kostka number $K_{\lambda,(r^n)}$ associated with the pair of partitions
$\lambda$ and $(r^n)$, that is, $g_{\lambda}^{(r)}$ is the coefficient of $m_{(r^n)}$ in the expansion of $s_\lambda$ in the monomial symmetric function basis $(m_{\mu})_{\mu}$ of $\C[\bt]$:
\be
\label{Kostka}
g_{\lambda}^{(r)} = \left[x_1^rx_2^r\cdots x_n^r\right]s_{\lambda}(t_1(\bx),t_2(\bx),\ldots),
\ee
where
\be
\label{eq:norm_power_sums}
t_j(\bx) = \frac{1}{j}\sum_{i=1}^{n}x_i^j\qquad j=1,2,\ldots
\ee
are the \emph{normalized power sum symmetric polynomials} in the variables $x_1,\ldots, x_n$.

In what follows, the coefficient representation \eqref{Kostka} of $g_{\lambda}^{(r)}$ will be rewritten as an $r$-fold integral of the Schur polynomial associated with $\lambda$ in terms of the normalized power sum variables. To proceed in this direction we need the following basic but important 

\begin{proposition}
\label{prop:coeff}
 Let $P(t_1,t_2,\ldots, t_r)$ be a polynomial in the indeterminates $t_1,\ldots, t_r$ and consider the polynomial
 \be
 P\left(t_1(\bx),t_2(\bx),\ldots,t_r(\bx)\right) \in \C[x_1,\ldots,x_n]
 \ee
  in the indeterminates $x_1,\ldots, x_n$, where $t_k(\bx)$ are the normalized power sum symmetric polynomials \eqref{eq:norm_power_sums} in $\bx$. The following integral identity holds:
\begin{multline}
\label{eq:coeff}
[x_1^r x_2^r\cdots x_n^r] P\left(t_1(\bx),t_2(\bx),\ldots,t_r(\bx)\right) \\
=\frac{1}{\pi^{r}}\underbrace{\int_{\C}\cdots\int_{\C}}_{r}P(u_1,\ldots, u_r)\overline{h_r\left(u_1,\frac{u_2}{2},\ldots, \frac{u_r}{r},0,0,\ldots\right)}^{n}\prod_{j=1}^{r}e^{-|u_j|^2}dA(u_j),
\end{multline}
where $h_r\left(u_1,\frac{u_2}{2},\ldots, \frac{u_r}{r},0,0,\ldots\right)$ is the complete symmetric polynomial of degree $r$ evaluated at $\left(u_1,\frac{u_2}{2},\ldots, \frac{u_r}{r},0,0,\ldots\right)$
and $dA$ is the area measure in the complex plane.
\end{proposition}
\begin{remark} Chapuy \cite{Ch} found a multiple integral formula for $g_{\lambda}^{(r)}$ involving $h_r^n$, expressed in terms of the original indeterminates $x_j$. Our identity \eqref{eq:coeff}, expressed in terms of the normalized power sum symmetric polynomials $t_k(\bx)$, has essentially the same combinatorial content. However, as it will be shown below, our choice of parametrization turns out to be more advantageous in finding a generating function for the numbers $A_{d+1,r}(n)$.
\end{remark}
\begin{remark}
The symmetric function $h_r^n$ appears in a closely related problem on the enumeration of a certain class of integer matrices: Everett and Stein \cite{ES} cite the following ``classical result'' attributed to MacMahon \cite{MacM}: the number of $n\times n$ matrices with non-negative integer entries whose row and column sums are all equal to $r$ is given by the coefficient of $m_{(r^n)}$ in the expansion of $h_r^n$ in the monomial symmetric function basis.
\end{remark}
\begin{proof}[Proof of Lemma \ref{prop:coeff}]
For $r$-tuples of non-negative integers $(k_1,k_2,\ldots, k_r)$, let
\be
a^{n,r}_{k_1,\ldots, k_r} =[x_1^r x_2^r\cdots x_n^r]t_1(\bx)^{k_1}\cdots t_r(\bx)^{k_r}.
\ee
The exponential generating function of the numbers $a^{n,r}_{k_1,\ldots, k_r}$ can be calculated as follows:
\begin{align}
\sum_{k_1,\ldots, k_r=0}^{\infty} a^{n,r}_{k_1,\ldots, k_r}\prod_{j=1}^{r}\frac{u_j^{k_j}}{k_j!}&=[x_1^r x_2^r\cdots x_n^r]\sum_{k_1,\ldots, k_r=0}^{\infty} \prod_{j=1}^{r}\frac{\left(t_j(\bx)u_j\right)^{k_j}}{k_j!}\cr
&=[x_1^r x_2^r\cdots x_n^r]\exp\left(\sum_{j=1}^{r}t_j(\bx)u_j\right)\cr
&=[x_1^r x_2^r\cdots x_n^r]\prod_{i=1}^{n}\exp\left(\sum_{j=1}^{r}\frac{x_i^j}{j}u_j\right)\cr
&=\prod_{i=1}^{n}[x_i^r]\exp\left(\sum_{j=1}^{r}\frac{x_i^j}{j}u_j\right)\cr
&=\left(h_r\left(u_1,\frac{u_2}{2},\ldots, \frac{u_r}{r},0,0,\ldots\right)\right)^{n}.
\end{align}
Since
\be
\frac{1}{\pi}\int_{\C}z^{k}\c{z}^{l}e^{-|z|^2}dA(z) = \delta_{kl}k!
\ee
for all $k,l \geq 0$, we have
\begin{multline}
\frac{1}{\pi^{r}}\underbrace{\int_{\C}\cdots\int_{\C}}_{r}u_1^{k_1}\cdots u_r^{k_r}\overline{\left(h_r\left(u_1,\frac{u_2}{2},\ldots, \frac{u_r}{r},0,0,\ldots\right)\right)^{n}}\prod_{j=1}^{r}e^{-|u_j|^2}dA(u_j)\\
=a^{n,r}_{k_1,\ldots, k_r} = [x_1^{r}x_2^{r}\cdots x_n^{r}]t_1(\bx)^{k_1}\cdots t_r(\bx)^{k_r}
\end{multline}
for all $r$-tuples $(k_1,k_2,\ldots, k_r)$,  and hence \eqref{eq:coeff} follows by taking linear combinations.
\end{proof}
\begin{corollary} For a partition of weight $|\lambda|=rn$ we have
\be
g^{(r)}_{\lambda}=\frac{1}{\pi^{r}}\underbrace{\int_{\C}\cdots\int_{\C}}_{r}s_{\lambda}(t_1,\ldots, t_r, 0,\ldots)\overline{\left(h_r\left(t_1,\frac{t_2}{2},\ldots, \frac{t_r}{r},0,\ldots\right)\right)^{n}}\prod_{j=1}^{r}e^{-|t_j|^2}dA(t_j).
\ee
\end{corollary}
Note that for a pair of weighted homogeneous polynomials $P$ and $Q$,
\be
\frac{1}{\pi^{r}}\underbrace{\int_{\C}\cdots\int_{\C}}_{r}P(t_1,\ldots,t_r)\overline{Q(t_1,\ldots,t_r)}\prod_{j=1}^{r}e^{-|t_j|^2}dA(t_j)=0 \quad \text{unless}\quad \deg P=\deg Q.
\ee

Since $s_{\lambda}$ and $h_r^{m}$ are weighted homogeneous polynomials with $\deg(s_{\lambda}) = |\lambda|$ and $\deg(h_r^n) = rm$ we have
\be
\label{pre_exponential}
\frac{1}{\pi^{r}}\underbrace{\int_{\C}\cdots\int_{\C}}_{r}s_{\lambda}(t_1,\ldots, t_r, 0,\ldots)\overline{\left(h_r\left(t_1,\frac{t_2}{2},\ldots, \frac{t_r}{r},0,\ldots\right)\right)^{m}}\prod_{j=1}^{r}e^{-|t_j|^2}dA(t_j)=\delta_{|\lambda|,rm}g_{\lambda}^{r}.
\ee
\begin{remark} For example, for $r=1,2,3$ we have
\begin{align*}
h_1\left(u_1,0,\ldots\right)&= u_1\\
h_2\left(u_1,\frac{u_2}{2},0,\ldots\right)&= \frac{1}{2}u_1^2+\frac{1}{2}u_2\\
h_3\left(u_1,\frac{u_2}{2},\frac{u_3}{3},0,\ldots\right)&= \frac{1}{6}u_1^3+\frac{1}{2}u_1^2u_2+\frac{1}{3}u_3.
\end{align*}
The general formula \cite[Eq.\ ($2.14^\prime$)]{Mac} is
\be
h_r\left(u_1,\frac{u_2}{2},\ldots,\frac{u_r}{r},0,\ldots\right)=\sum_{|\mu|=r} \frac{1}{z_{\mu}}\prod_{i=1}^{\ell(\mu)} u_{\mu_i},
\ee
where $z_{\mu}$ is the standard constant factor
\be
z_{\mu} = \prod_{i\geq 1}i^{m_i}\cdot m_i!
\ee
in which $m_i=m_i(\mu)$ is number of parts of the partition $\mu$ equal to $i$. Note that
\be
\left|C_{\mu}\right| = \frac{r!}{z_{\mu}} \quad \text{and} \quad |\mu|=r,
\ee 
where $C_{\mu}$ is the conjugacy class of permutations of cycle type $\mu$ in the symmetric group $S_r$.
Hence $h_r\left(u_1,\frac{u_2}{2},\ldots, \frac{u_r}{r},0,0,\ldots\right)$ can be written alternatively as the cycle index polynomial of $S_r$:
\be
\label{eq:cycle_index}
h_r\left(u_1,\frac{u_2}{2},\ldots,\frac{u_r}{r},0,\ldots\right)= \frac{1}{r!}\sum_{|\lambda|=r} \left|C_{\lambda}\right|\prod_{i=1}^{\ell(\lambda)} u_{\lambda_i} = \Cyc_r(u_1,u_2,\ldots, u_r).
\ee
\end{remark}

By taking into account \eqref{pre_exponential} and \eqref{eq:cycle_index} we obtain the following:
\begin{corollary}
\be
\frac{1}{\pi^{r}}\underbrace{\int_{\C}\cdots\int_{\C}}_{r}s_{\lambda}(\bt)\exp\left(\bar{v}^r\overline{\Cyc_{r}(\bt)}\right)\prod_{j=1}^{r}e^{-|t_j|^2}dA(t_j) 
=
\begin{cases}
\dfrac{\bar{v}^{rn}}{n!}g_{\lambda}^{(r)}, & \text{if $|\lambda|=rn$;}\\
&\\
0, & \textrm{otherwise.}
\end{cases}
\label{cropping}
\ee
\end{corollary}
By the RSK formula \eqref{RSK} this leads to an integral representation of $A_{d+1,r}(n)$:
\begin{lemma} The following identity holds for all non-negative integers $d$, $r$, and $n$:
\be
\frac{A_{d+1,r}(n)}{(rn)! n!}(v\c{v})^{rn}=\\ 
\frac{1}{\pi^{r}}\underbrace{\int_{\C}\cdots\int_{\C}}_{r}\left(\sum_{\substack{|\lambda|=rn\\ \ell(\lambda)\leq d}}s_{\lambda}(v,0,\ldots)s_{\lambda}(\bt)\right)
\exp\left(\bar{v}^r\overline{\Cyc_r(\bt)}\right)\prod_{j=1}^{r}e^{-|t_j|^2}dA(t_j).
\label{int_A}
\ee
\end{lemma}
\begin{proof} By using the scaled version of \eqref{f_lambda_schur} of the form
\be
\frac{f_{\lambda}}{|\lambda|!}v^{|\lambda|} = s_{\lambda}(v,0,0,\ldots),
\ee
we obtain
\begin{align}
\frac{A_{d+1,r}(n)}{(rn)! n!}(v\c{v})^{rn}
&=\sum_{\substack{|\lambda|=rn\\ \ell(\lambda)\leq d}}\frac{f_{\lambda}}{|\lambda|!}v^{|\lambda|}\cdot \frac{g_{\lambda}^{(r)}}{n!}\bar{v}^{rn}\cr
&=\sum_{\substack{|\lambda|=rn\\ \ell(\lambda)\leq d}}s_{\lambda}(v,0,\ldots)\cdot \frac{1}{\pi^{r}}\underbrace{\int_{\C}\cdots\int_{\C}}_{r}s_{\lambda}(\bt)\exp\left(\bar{v}^r\overline{\Cyc_r(\bt)}\right)\prod_{j=1}^{r}e^{-|t_j|^2}dA(t_j)\cr
&=\frac{1}{\pi^{r}}\underbrace{\int_{\C}\! \cdots\! \int_{\C}}_{r}\left(\sum_{\substack{|\lambda|=rn\\ \ell(\lambda)\leq d}}s_{\lambda}(v,0,\ldots)s_{\lambda}(\bt)\right)
\exp\left(\bar{v}^r\overline{\Cyc_r(\bt)}\right)\prod_{j=1}^{r}e^{-|t_j|^2}dA(t_j).\cr
\end{align}
\end{proof}

\section{Gessel's Toeplitz determinant identity and the proof of the main theorem}
\label{sec:gessel}

Recall Gessel's determinant formula \cite{Ge}: the length-restricted Cauchy--Littlewood-type sum of products of Schur functions can be written as
\be
\label{Gessel_det}
\sum_{\ell(\lambda)\leq d} s_{\lambda}({\mathbf t})s_{\lambda}({\mathbf s})= \det(G_d(\bt,\bs)),
\ee
where $\bt=(t_1,t_2,\ldots), {\mathbf s}=(s_1,s_2,\ldots)$, and $G_d(\bt,\bs)$ is the $d\times d$ Toeplitz determinant associated with the symbol
\be
\varphi(\bt,\bs; z)=\exp\left(\sum_{k=1}^{\infty} t_k z^{k}+\sum_{k=1}^{\infty}s_k z^{-k}\right)=\sum_{m=-\infty}^{\infty} \varphi_{m}({\mathbf t},{\mathbf s}) z^{m}.
\ee
This ``master identity'', combined with the integral representation for $A_{d+1,r}(n)$ is sufficient to prove the main result of this paper.
\begin{proof}[Proof of Theorem~\ref{thethm}] Gessel's determinantal formula \eqref{Gessel_det} with
\be
\bt =(t_1,\ldots,t_r,0,,0\ldots) \quad \text{and}\quad \bs=(v,0,0,\ldots)
\ee
gives
\be
\sum_{\ell(\lambda)\leq d}s_{\lambda}(v,0,\ldots)s_{\lambda}(t_1,\ldots, t_r, 0,\ldots) = \det(T_d(v,\bt))
\ee
where $T_d(v,\bt)$ is the $d\times d$ Toeplitz determinant associated with the symbol
\be
\varphi(v,\bt;z)=\sum_{m=-\infty}^{\infty}\varphi_{m}(v,\bt)z^m=\exp\left(\frac{v}{z}+ \sum_{j=1}^{r}t_j z^j\right).
\ee
Therefore
\begin{align}
&\frac{1}{\pi^{r}}\underbrace{\int_{\C}\cdots\int_{\C}}_{r}\det(T_d(v,\bt))\cdot \exp\left(\bar{v}^r\overline{\Cyc_{r}(\bt)}\right)\prod_{j=1}^{r}e^{-|t_j|^2}dA(t_j)\cr
&=\frac{1}{\pi^{r}}\underbrace{\int_{\C}\cdots\int_{\C}}_{r}\left(\sum_{\ell(\lambda)\leq d}s_{\lambda}(v,0,\ldots)s_{\lambda}(t_1,\ldots, t_r, 0,\ldots)\right)\cdot \exp\left(\bar{v}^r\overline{\Cyc_{r}(\bt)}\right)\prod_{j=1}^{r}e^{-|t_j|^2}dA(t_j)\cr
\label{lastline}
&=\frac{1}{\pi^{r}}\underbrace{\int_{\C}\cdots\int_{\C}}_{r}\sum_{n=0}^{\infty}\left(\sum_{\substack{\ell(\lambda)\leq d\\|\lambda|=rn}}s_{\lambda}(v,0,\ldots)s_{\lambda}(t_1,\ldots, t_r, 0,\ldots)\right)\cdot \exp\left(\bar{v}^r\overline{\Cyc_{r}(\bt)}\right)\prod_{j=1}^{r}e^{-|t_j|^2}dA(t_j),\cr
\end{align}
where in the last step we used the simple observation that, by \eqref{cropping}, a partition $\lambda$ gives a non-trivial contribution to the integral only if its weight $|\lambda|$ is divisible by $r$. The proof is concluded by comparing \eqref{lastline} with the integral representation \eqref{int_A} of $A_{d+1,r}(n)$.
\end{proof}

\section{Calculation of $A_{d+1,r}(n)$ using symbolic computer algebra}
\label{sec:calculations}
The purpose of this section is to illustrate that the generating function \eqref{main_formula} can be used to calculate $A_{d+1,r}(n)$ for small values of $d$, $r$, and $n$ using basic symbolic computer algebra. There was no attempt made to optimize the algorithms used in the symbolic computation; our goal was simply to demonstrate that the generating function is not only a compact and abstract formula, but it can also be put to work in practice. This also provides numerical evidence that the series \eqref{main_formula} does give the integer values it is supposed to, which is, after all, the most rewarding part of finding a generating function.

Below we use the short-hand notation $\bt_r = (t_1,\ldots, t_r)$. It is easy to see that the Laurent series coefficients of the Toeplitz symbol $\varphi(v,\bt_r;z)$ are
\begin{align}
\varphi_{m}(v,\bt_r) &= \sum_{k=0}^{\infty}h_{k+m}(\bt_r)\frac{v^k}{k!}\quad m=0,1,2,\ldots\\
\varphi_{-m}(v,\bt_r) &= \sum_{k=m}^{\infty}h_{k-m}(\bt_r)\frac{v^k}{k!}\quad m=1,2,\ldots,
\end{align}
where $h_k(\bt_r)$ are the complete symmetric polynomials in $\bt_r$.
We can define the truncated Toeplitz entries
\begin{align}
\varphi^{(N)}_{m}(v,\bt_r) &= \sum_{k=0}^{rN}h_{k+m}(\bt_r)\frac{v^k}{k!}\quad m=0,1,2,\ldots\\
\varphi^{(N)}_{-m}(v,\bt_r) &= \sum_{k=m}^{rN}h_{k-m}(\bt_r)\frac{v^k}{k!}\quad m=1,2,\ldots
\end{align}
so that
\be
\varphi_{m}(v,\bt_r) = \varphi^{(N)}_{m}(v,\bt_r) +\order{v^{rN+1}}.
\ee
For small values of $d$ the Toeplitz determinant with $\order{v^{rN+1}}$-truncated entries can be evaluated, and it can be further simplified by eliminating the terms of $\order{v^{rN+1}}$. We note that the evaluation of the determinant can be done much more efficiently by employing the division-free algorithm suggested by Ekhad, Shar, and Zeilberger \cite{ESZ}.

The expansion of the second factor $e^{\c{v}^r\Cyc_t(\c{t}_1,\ldots,\c{t}_r)}$ in the integral formula \eqref{main_formula} in $\c{v}$ is of the form
\be
\exp\left(\c{v}^r\Cyc_t(\c{t}_1,\ldots,\c{t}_r)\right)
=E_r^{(N)}(\bar{t}_1,\bar{t}_2/2,\ldots, \bar{t}_r/r;\bar{v})+\order{\bar{v}^{rN+r}},
\ee
where
\be
E_r^{(N)}(\bar{t}_1,\bar{t}_2/2,\ldots, \bar{t}_r/r;\bar{v})=\sum_{k=0}^{N} h_r(\bar{t}_1,\bar{t}_2/2,\ldots, \bar{t}_r/r,0,0)^k \frac{\bar{v}^{kr}}{k!}.
\ee
Therefore, by replacing the Toeplitz determinant and the exponential factor with their truncated forms in the integral \eqref{main_formula}, we get
\begin{multline}
 \frac{1}{\pi^{r}}\underbrace{\int_{\C}\cdots\int_{\C}}_{r}\det\left(T^{(N)}_{d}(v,\bt_r)\right)
E_r^{(N)}(\bar{t}_1,\bar{t}_2/2,\ldots, \bar{t}_r/r;\bar{v})\prod_{j=1}^{r}e^{-|t_j|^2}dA(t_j)\\
=\sum_{n=0}^{N}\frac{A_{d+1,r}(n)}{(rn)!n!}(v\c{v})^{rn} +\order{(v\c{v})^{rN+r}}.
\end{multline}
That is, the terms of order at most $rN$ in the product $v\c{v}$ can be trusted from the integral of the truncated expression.

The computation of area integrals can be implemented easily by using polar coordinates
\be
\frac{1}{\pi}\int_{\C}F(t)e^{-|t|^2}dA(t) = 2\int_{0}^{\infty}\left(\frac{1}{2\pi i}\oint_{|u|=1}F(ru)\frac{du}{u}\right)e^{-r^2}rdr.
\ee
The contour integrals are computed by taking residues at $u=0$ (as $F(t)$ in the integrand is always a polynomial in $t$ in our setting).

This simple method was implemented in Maple, and we obtained the values shown in Tables \ref{tab:r=2}, \ref{tab:r=3}, and  \ref{tab:r=4}. These all coincide with the numbers presented in the various appendices to Zeilberger's online journal entry \cite{EZ}.

\begin{table}[htb!]
\begin{center}
\small
\begin{tabular}{c||c|c|c|c|c}
$n\backslash d$&$1$ & $2$ &$3$ &$4$ &$5$\\
\hline
\hline
1 & 1&1&1&1&1\\
2 & 1&6&6&6&6\\ 
3& 1&43&90&90&90\\ 
4& 1&352&1879&2520&2520\\ 
5& 1&3114&47024&102011&113400\\
6& 1&29004&1331664&5176504&7235651\\
7& 1&280221&41250519&307027744&592616287\\ 
8& 1&2782476&1367533365&20472135280&58255807971\\ 
9& 1&28221784&47808569835&1496594831506&6585311137855\\ 
10& 1&291138856&1744233181074&117857270562568&832218817076725
\end{tabular}
\end{center}
\caption{The calculated values of $A_{d+1,2}(n)$.}
\label{tab:r=2}
\end{table}

\begin{table}[htb!]
\begin{center}
\begin{tabular}{c||c|c|c|c|c}
$n \backslash d$ &$1$ & $2$ &$3$ &$4$ &$5$\\
\hline
\hline
1 & 1&1&1&1&1\\
2 & 1&20&20&20&20\\
3 & 1&374&1680&1680&1680\\
4 & 1&8124&173891&369600&369600\\ 
5 & 1&190893&21347262&117392909&168168000\\ 
6 & 1&4727788&2977892253&46121962742&117108036719\\
7 & 1&121543500&455912368540&21198300356500&105795227339731\\
8 & 1&3212914524&74876841353159&11003612776114008&115061550940847029
\end{tabular}
\end{center}
\caption{The calculated values of $A_{d+1,3}(n)$.}
\label{tab:r=3}
\end{table}

\begin{table}[htb!]
\begin{center}
\begin{tabular}{c||c|c|c|c}

$n\backslash d$ &$1$ & $2$ &$3$ &$4$\\
\hline
1 &1&1&1&1\\ 
2 & 1&70&70&70\\ 
3 & 1&3199&34650&34650\\
4 & 1&173860&16140983&63063000\\ 
5 & 1&10203181&8854463421&142951955371\\ 
6 & 1&631326526&5532980565456&389426248416626
\end{tabular}
\end{center}
\caption{The calculated values of $A_{d+1,4}(n)$.}
\label{tab:r=4}
\end{table}

\section{Conclusion and outlook}
In this paper we presented the generating function \eqref{main_formula} to solve the challenge posed by Ekhad and Zeilberger \cite{EZ} to generalize Gessel's determinant formula \eqref{Gessel_original} for $r=1$ to arbitrary values of the multiplicity $r$. The key step in constructing such a generating function was to express the Kostka coefficient $g_{\lambda}^{(r)}$ as an integral of the Schur function $s_\lambda$, and use Gessel's ``master identity'' to sum up the resulting Cauchy--Littlewood-type terms under the integrals.

We expect that the general formula \eqref{main_formula} may be useful to find the associated linear differential equation with  polynomial coefficients for the generating function, or equivalently, to find  the recurrence relation with polynomial coefficients for $A_{d+1,r}(n)$ for arbitrary values of $r$ in a closed form.

Moreover, one can easily construct generating functions analogous to \eqref{main_formula} for words with different types of multiple occurrences in its letters.

It may also be interesting to investigate the asymptotic distribution obtained from a fine-tuned scaling limit of the probability $\operatorname{Prob}\left(L_{n,r}(w)\leq d\right)$ as $n,d \to \infty$.

\section{Acknowledgments}
Many thanks to Doron Zeilberger for posting the challenge of generalizing Gessel's generating function, his encouragement and kind feedback on a previous version of this manuscript, along with a generous donation to the OEIS in my name (see \url{https://oeisf.org/donate/#2015}). 
I am grateful to Marco Bertola for his interest in this project and helpful discussions, and, in particular, for pointing out the standard scalar product interpretation \eqref{main_formula_scalar} of the integral formulae. 
The thoughtful comments and suggestions of the editor and the anonymous referee helped a lot to improve the manuscript and made the paper much easier to read.

This work was supported in part by the \emph{Fonds de recherche du Qu\'ebec --- Nature et technologies} (FRQNT) via the research grant \emph{Matrices Al\'eatoires, Processus Stochastiques et Syst\`emes Int\'egrables}. The author acknowledges the excellent working conditions at Concordia University and the Centre des Recherches Math\'ematiques (CRM) in Montreal.

2020 Mathematics Subject Classification: Primary 05A15; Secondary 05E05, 05A05.\\

Keywords: longest increasing subsequence, random permutation of multiset, random word, Schur function, Toeplitz determinant\\

Concerned with the following sequences:\\
\seqnum{A220097}, \seqnum{A266734}, \seqnum{A266735}, \seqnum{A266737},  \seqnum{A266741}, \seqnum{A267479}, \seqnum{A267480}.

\end{document}